\documentclass[12pt]{amsart}
\usepackage{amssymb,amsmath,graphics,verbatim}
\usepackage{latexsym}
\usepackage{eucal}

\usepackage{a4wide}

\usepackage{color}

\newtheorem{theorem}{Theorem}
\newtheorem{lemma}[theorem]{Lemma}
\newtheorem{prop}[theorem]{Proposition}

\theoremstyle{definition}

\theoremstyle{remark}
\newtheorem{remark}[theorem]{Remark}

\newcommand{\bR}{\mathbb{R}}
\newcommand{\bH}{\mathbb{H}}
\newcommand{\cD}{\mathcal{D}}
\newcommand{\cF}{\mathcal{F}}
\newcommand{\cM}{\mathcal{M}}
\newcommand{\cT}{\mathcal{T}}
\newcommand{\dg}{\dot{g}}
\newcommand{\doh}{\dot{h}}
\newcommand{\dA}{\dot{A}}

\newcommand{\Diff}{\mathrm{Diff}}

\newcommand{\geod}{\mathrm{geod}}

\newcommand{\cg}{\mathfrak{g}}
\newcommand{\vol}{\mathrm{vol}}
\newcommand{\Vol}{\mathrm{Vol}}

\newcommand{\sym}{\mathrm{sym}}
\newcommand{\Tr}{\mathrm{Tr}}

\newcommand{\Volr}{\Vol_R}
\newcommand{\Ric}{\operatorname{Ric}}
\newcommand{\circR}{\mathring{R}}

\begin{document}
\title[Convexity of the renormalized volume]{Convexity of the renormalized volume of hyperbolic $3$-manifolds}
\author{Sergiu Moroianu}
\thanks{Partially supported by the CNCS project PN-II-RU-TE-2012-3-0492.}
\address{Sergiu Moroianu, Institutul de Matematic\u{a} al Academiei Rom\^{a}ne\\ P.O. Box 1-764\\ RO-014700 Bucharest\\
Romania}
\email{moroianu@alum.mit.edu}
\date{\today}
\begin{abstract}
The Hessian of the renormalized volume of geometrically finite hyperbolic 
$3$-manifolds without rank-$1$ cusps, computed at the hyperbolic metric $g_\geod$
with totally geodesic boundary of the convex core, is shown to be a strictly 
positive bilinear form on the tangent space to Teichm\"uller space. 
The metric $g_\geod$ is known from results of Bonahon and Storm to be an absolute minimum 
for the volume of the convex core.
We deduce the strict convexity of the functional volume of the convex core at its minimum point.
\end{abstract}
\maketitle

\section{Introduction}

The renormalized volume is a functional on the moduli space of 
hyperbolic $3$-mani\-folds of finite geometry. It has been introduced 
in this context by Krasnov \cite{Kr}, after initial work by Henningson 
and Skenderis \cite{hs} for more general Poincar\'e-Einstein manifolds. 
As $3$-dimensional geometrically finite $3$-manifolds are closely 
related to Riemann surfaces,  $\Volr$ defines in a natural way
a K\"ahler potential for the Weil-Petersson symplectic form on 
the Teichm\"uller space. This follows for quasi-fuchsian manifolds
by the identity between the renormalized volume and the so-called 
classical Liouville action functional, a topological quantity known 
by work of Takhtadzhyan and
Zograf \cite{TaZo} to provide a K\"ahler potential. For geometrically 
finite hyperbolic $3$-manifolds without rank-$1$ cusps, the 
K\"ahler property 
of the renormalized volume was proved by Colin Guillarmou and 
the author in \cite{CS}, by constructing a Chern-Simons theory 
on the Teichm\"uller space. The case of cusps of rank $1$ is
studied in a joint upcoming paper with Guillarmou and Fr\'ed\'eric 
Rochon.

Here we look at a certain moduli space of complete, infinite-volume hyperbolic 
metrics $g$ on a fixed $3$-manifold $X$. The metrics we consider are
geometrically finite quotients $\Gamma\backslash \bH^3$ 
(i.e., they admit a fundamental polyhedron with 
finitely many faces) and do not have cusps of rank $1$, in the sense that
every parabolic subgroup of $\Gamma$, if any, must have rank $2$. 
We define the moduli space $\cM$ as the quotient of the above set of metrics
on $X$ by the group $\Diff^0(X)$ of diffeomorphisms isotopic to the identity.
The existence of such metrics on $X$ implies that $X$ is
diffeomorphic to the interior of a manifold-with-boundary $K$. 
Let $2K$ be the smooth manifold obtained by doubling $K$ across $\Sigma$. 
We make the following assumption throughout the paper: 
\begin{quote}
\emph{There exists on $2K$ a complete hyperbolic metric 
of finite volume.}
\end{quote}
It follows from Mostow-Prasad rigidity that 
up to a diffeomorphism of $2K$ isotopic to the identity, the boundary $\Sigma=\partial K$
is totally geodesic for this metric. Since $2K$ must be aspherical 
 and atoroidal, the connected components of $\Sigma$ 
cannot be spheres or tori.

Examples of manifolds where our 
assumption is \emph{not} fulfilled are quasi-fuchsian manifolds 
and Schottky manifolds, since their 
double is not atoroidal. With the above assumption, 
a distinguished point $g_\geod$ in $\cM$ is obtained 
from $K$ by gluing infinite-volume \emph{funnels} with vanishing Weingarten
operator (see Section \ref{sf}) 
to each boundary component of $K$. We call this metric the 
\emph{totally geodesic metric}, 
and note that $K$ is the convex core of $(X,g_\geod)$. It was remarked 
by Thurston, again as a simple consequence of Mostow rigidity, 
that $g_\geod$ is the unique metric in $\cM$ with smooth boundary 
of the convex core.

By work of Bonahon \cite{bon} it is known that the volume of the convex 
core $\Vol(C(X,g))$ has a minimum at $g_\geod$ when viewed as a 
functional on $\cM$. When $X$ is convex co-compact, i.e., without cusps,
Storm \cite{storm} proved that the minimum point $g_\geod$ is strict.
We shall apply here our results on $\Volr$ to deduce the 
convexity of $\Vol(C(X,g))$ at this special point in $\cM$ for $X$ 
geometrically finite without cusps of rank $1$, but possibly with 
cusps of rank $2$ as in \cite{bon}.

It is instructive to compare those results to the situation for 
quasi-fuchsian manifolds. Combining results of Schlenker \cite{Sch} 
and Brock \cite{Brk}, the renormalized volume of quasi-fuchsian manifold 
is commensurable on Teichm\"uller space to the volume of the convex core.
In particular, it is \emph{not} proper as a function on 
Teichm\"uller space, since it remains bounded under iterations 
of a Dehn twist. 
It has been stated without proof by Krasnov and Schlenker
\cite{KS08} that $\Volr$ is non-negative on the quasi-fuchsian space. 
There is some compelling evidence for this claim: it was proved in
\cite{KS08} that the only critical point in a Bers slice is 
at the fuchsian locus, and there the Hessian of the renormalized 
volume equals a multiple of the Weil-Petersson scalar product. 
However, the lack 
of properness does not allow one to conclude that $\Volr$ is 
globally non-negative. In a recent joint paper with Corina Ciobotaru, 
we proved that $\Volr$
is non-negative on the \emph{almost-fuchsian} space, an open 
neighborhood of the fuchsian locus inside the quasi-fuchsian 
space, and that it vanishes there only at the fuchsian locus.

Schlenker's results from \cite{Sch} have been recently 
extended to convex co-compact hyperbolic $3$-manifolds 
by Bridgeman and Canary \cite{BC}. They obtain quite nice global
results bounding the renormalized volume in terms of the convex core.

The main result of this paper describes the local behavior of $\Volr$ near $g_\geod$.
\begin{theorem}\label{main}
Let $g_\geod$ be a geometrically finite hyperbolic metric on $X$ without 
rank $1$-cusps and with totally geodesic boundary of 
the convex core. Then the Hessian of the renormalized volume functional 
on $\cM$ at $g_\geod$ is positive definite.
\end{theorem}
The proof is done in two steps. First we look at the volume 
enclosed by minimal surfaces 
near the boundary of the convex core, proving that it is convex, 
and then compare it to the ``optimal'' renormalization with respect to the unique 
hyperbolic metric in the conformal class at infinity. 
In the first step we use a boundary-value problem for the linearized Einstein 
equation in Bianchi gauge at a metric with geodesic boundary. 
The Hessian of the volume appears as a Dirichlet-to-Neumann operator, 
which we prove to be strictly positive by an appropriate Weitzenb\"ock formula. 
The second step uses the analysis of the uniformizing conformal factor,
together with some elementary elliptic theory.

As a consequence, we obtain the convexity of the convex core functional:

\begin{theorem}\label{coro}
Let $(X,g_\geod)$ be a geometrically finite hyperbolic $3$-manifold 
with totally geodesic boundary of the convex core, and without 
rank-$1$ cusps. Then for metrics $g\in\cM$ near $g_\geod$, the Hessian 
at $g_\geod$ of the functional
\begin{align*}
\cM\longrightarrow\bR,&&g\longmapsto\Vol(C(X,g))
\end{align*}
is positive definite as a bilinear form on $T_{g_\geod}\cM$.
\end{theorem}

By the simultaneous uniformization result of Ahlfors and Bers 
valid for quasi-fuchsian manifolds, extended by Marden \cite{Marden} to the 
geometrically finite case, $\cM$ is identified with the Teichm\"uller space of 
$\Sigma$, keeping in mind that the connected components of $\Sigma$ 
have genus at least $2$. We identify therefore $T\cM$ with $T\cT_\Sigma$. 
Our strategy in proving Theorem \ref{main} will be to bound from below the Hessian of the 
renormalized volume in terms of the Weil-Petersson metric on $\cT_\Sigma$.

\subsection*{Acknowledgments}
This paper originated from discussions with Colin Guillarmou 
and Jean-Marc Schlenker 
about the renormalized volume of hyperbolic $3$-mani\-folds. 
I owe in particular to Jean-Marc the observation
that Theorem \ref{main} has implications 
about the volume of the convex core. 

\section{Funnels}\label{sf}
Let $(X,g)$ be a geometrically finite hyperbolic $3$-manifold without rank-$1$ 
cusps. Such a manifold can be decomposed in a finite-volume part $K$ 
(a smooth manifold-with-boundary with a finite number of cusps of rank $2$), 
and a finite number of funnels. These funnels play an important role in this paper, 
so we review them below.

A \emph{funnel} is a hyperbolic half-cylinder $(F,g)$, where 
$F=[0,\infty)\times\Sigma$, for some compact, 
possibly disconnected Riemannian surface $(\Sigma,h)$, while
\begin{align}\label{fun}
g=dt^2+h_t,&& h_t=h\left((\cosh t +A\sinh t)^2\cdot,\cdot \right).
\end{align}
Here $A$ is a symmetric field of endomorphisms of $T\Sigma$, namely the Weingarten operator of the isometric inclusion 
$\{0\}\times \Sigma\hookrightarrow F$. The tensors $h_t$ are Riemannian metrics on $\Sigma$ whenever $t$ is such that the eigenvalues of $A$
are larger than $-\coth t$. Hence, we must assume that $A+1$ is positive definite in order for $g$ to be well-defined on the whole half-cylinder.
For notational simplicity, we allow disconnected funnels.

The necessary and sufficient conditions for $g$ to be hyperbolic are the hyperbolic version of the Gauss and Codazzi--Mainardi equations:
\begin{align}
&\det(A)=\kappa_h+1, \label{gauss} \\
&(d^\nabla)^* A+d\Tr(A)=0\label{codazzi}
\end{align}
where $\kappa_h$ is the Gaussian curvature of $h$.
Let 
\[H:=\Tr(A):\Sigma\to \bR\] 
be the mean curvature function (without the customary $1/2$ factor) of $\{0\}\times \Sigma\hookrightarrow X$ 
with respect to the direction $\partial_t$ escaping from $K$.
Let $A_t, H_t,\kappa_{h_t}$ be the Weingarten map, the mean curvature, 
respectively the Gaussian curvature, of $\{t\}\times \Sigma\hookrightarrow F$.
We have 
\[A_t=\tfrac{1}{2}g^{-1}\partial_t g=(\cosh t +A\sinh t)^{-1}(\sinh t+A\cosh t)\]
The Gauss and Codazzi-Mainardi equations continue to hold at every $t$, so 
\[\kappa_{h_t}=\det(A_t)-1.\]

\section{Renormalized volumes}

For a geometrically finite metric without rank-$1$ cusps with a 
fixed funnel struc\-tu\-re, we define the induced metric at infinity
on the surface $\Sigma$:
\[h_\infty:=\lim e^{-2t} h_t=\tfrac{1}{4}h\left((1+A)^2\cdot,\cdot \right).\]
The renormalized volume of $(X,g)$ with respect to $h_\infty$ is defined by 
Krasnov and Schlen\-ker \cite{KS08} as
\[\Volr(X,g;h_\infty)=\Vol(K,g)-\tfrac14 \int_\Sigma H dh\]
where $H$ is the trace of $A$. This is the same as the Riesz-regularized 
volume with respect to the boundary-defining function $e^{-t}$, see e.g.\ \cite{GMS},. 
Let $\omega\in C^\infty(\Sigma)$ be the unique conformal factor such that the metric
$e^{2\omega}h_\infty$ is of constant curvature equal to $-4$. Like every metric in the conformal class of 
$h_\infty$, the metric $e^{2\omega}h_\infty$ arises as the metric at infinity for some other funnel structure 
on $(X,g)$. The \emph{renormalized volume of $(X,g)$} is defined (cf.\ Krasnov \cite{Kr})
with respect to this canonical choice:
\[\Volr(X,g):=\Volr(X,g;e^{2\omega}h_\infty).\]
A proof of this equality appears for instance explicitly in \cite{CM}, Prop.\ 5. 
Note that the chosen metric at infinity is not hyperbolic but of curvature $-4$.
so the volume of $(M,e^{2\omega}h_\infty$ equals $\pi(g-1)$. For every other metric 
$h'$ conformal to $h$ and of the same volume, $\Volr(X,g;h')<\Volr(X,g;e^{2\omega}h_\infty)$,
hence the above choice is very natural.

\section{Geometrically finite manifolds with totally geodesic boundary}\label{cccm}

Let $g$ be a hyperbolic metric on $X$ such that the convex core of $X$ has totally geodesic 
boundary, denoted $\Sigma$. Then $A=0$, $H=0$, $h$ is hyperbolic and the induced metric 
at infinity 
\[h_\infty=\lim e^{-2t} h_t=h/4\]
has constant Gaussian curvature equal to $-4$. 
Let $\{g^s\}_{s\in\bR}$ be a one-parameter family of deformations of $g$ inside 
the space of convex co-compact metrics. (This space is parametrized by the deformations 
$[h_\infty^s]$ of the conformal classes
of the induced metrics $h_\infty^s$ at infinity). 

\begin{prop}
For small deformations $g^s$ of $g$, in the homotopy class of $\Sigma$ there exists a unique 
family of surfaces $\Sigma_s$ which are minimal for $g^s$.
\end{prop}
\begin{proof}
Note that $\pi_1(\Sigma)$ does not necessarily inject into $\pi_1(X)$.
For each connected component $\Sigma_j$ of $\Sigma$ cut out the funnel containing it and 
complete it to a quasi-fuchsian manifold (this is possible since for small enough $s$, 
the eigenvalues of $A_s$ are close to $0$). In that quasi-fuchsian manifold it is known e.g.\ by 
Uhlenbeck \cite{Uhl} that there exists a unique minimal surface homotopic to $\Sigma_j$. 
This surface will live in the original funnel of $X$ for small enough $s$.
\end{proof}

For $s$ close to $0$ let therefore $\Sigma_s\subset X$ be the unique minimal surface inside 
$X$ homotopic to $\Sigma$.
Choose $\{\Phi_s\}$ a family of diffeomorphisms of $X$ mapping $\Sigma$ onto $\Sigma_s$, with $\Phi_0$ equal to the identity, and furthermore such that
the outgoing geodesics on $\Sigma$ are mapped isometrically on the corresponding geodesics normal to $\Sigma_s$ with respect to the 
metric $g^s$. 
Hence, by pulling back $g^s$ via $\Phi_s$ we may assume that $[0,\infty)\times \Sigma$ 
is the underlying space of every funnel in the family $g^s$, of course with different 
metric $h^s$ and Weingarten map $A^s$. 

By composing with an additional family of diffeomorphisms preserving the surface $\Sigma$ and the funnel structure, we can further assume that
the first-order variation $\doh$ of the metrics $h^s$ induced by $g^s$ on $\Sigma$
(the dot on top of some tensor denotes $s$-derivative at $s=0$) is divergence-free:
\[\delta^h\doh=0.\]
Let $\dA$ be the first-order variation of the Weingarten map.

\begin{lemma}\label{lTT}
The tensors $\dot{h}$ and $\dot{A}$ along $\Sigma$ are trace- and divergence-free.
\end{lemma}
\begin{proof}
Since $h^s$ is minimal, we have 
$\Tr(A^s)=0$, and hence $\Tr(\dot{A})=0$. 
On one hand, differentiating the Gauss equation implies that the variation of the curvature of $h^s$ is $0$:
\[\frac{\partial}{\partial s} \kappa(h^s)_{|s=0}=\Tr(\dot{A})=0.\]
On the other hand, the variation of $\kappa$ at the hyperbolic metric $h$ 
is given by the following intrinsic formula (cf.\ e.g. \cite[Theorem 1.174(e)]{besse}):
\begin{equation}\label{vark}
2\dot{\kappa}=(\Delta_h+1)\Tr(\dot{h})+d^*\delta^h \dot{h}.
\end{equation}
Since $\delta^h \dot{h}=0$ and $\dot{\kappa}=0$, it follows by positivity of the elliptic operator $\Delta_h+1$ that $\Tr(\dot{h})=0$.
Differentiating the Codazzi equation $\delta^{h^s} A^s=0$ shows, since $A=0$, that $\delta^h \dot{A}=0$. 
\end{proof}

\begin{lemma}\label{gTT}
The tensor $\dot{g}$ on $X$ is trace- and divergence-free in a neighborhood of $\Sigma$ containing the funnel.
\end{lemma}
\begin{proof}
We have 
\[\dg=\cosh^2(t)\doh+2\sinh(t)\cosh(t)h(\dA\cdot,\cdot).
\]
This tensor is clearly trace-free, since $\dA,\doh$ are trace-free. 

Let $T:=\partial_t$ (we denote by $\nu$ the restriction of $T$ along $\Sigma$). Write $f=\cosh(t)$ so that
\begin{align*}
g=dt^2+f^2h,&& \dg=f^2\doh+2ff'h(\dA\cdot,\cdot).
\end{align*}
For every tangential vector fields 
$U,V$ independent of $t$, i.e., such that $[T,U]=[T,V]=0$, we have directly from the Koszul formula
\begin{align*}
\nabla_TT=0,&&\nabla_TU=\tfrac{f'}{f}U=\nabla_UT,&&\nabla_UV=\nabla^\Sigma_UV-ff'h(U,V)T.
\end{align*}
For a $1$-forms $\alpha$ with $L_T\alpha=0$ and $\alpha(T)=0$, we get by duality
\begin{align*}
\nabla_Tdt=0,&&\nabla_T\alpha=-\tfrac{f'}{f}\alpha,&&\nabla_U\alpha=\nabla^\Sigma_U\alpha-ff'\alpha(U)dt.
\end{align*}

Since $\dg(T,\cdot)=0$ it follows from the above table that $(\nabla_T\dg)(T,\cdot)=0$. Moreover, if $\{e_1,e_2\}$
is an orthonormal frame for $h$, then 
\begin{align*}
-\sum_{i=1}^2 (\nabla_{e_i}\dg)(e_i,e_j)=-\sum_{i=1}^2 (\nabla^\Sigma_{e_i}\dg)(e_i,e_j),&&-\sum_{i=1}^2 (\nabla_{e_i}\dg)(e_i,T)=ff'\Tr(\dg).
\end{align*}
Both these terms vanish by Lemma \ref{lTT}.
\end{proof}

\section{Variation of the Einstein equation}

In dimension $3$, a hyperbolic metric means an Einstein metric with constant $-2$:
\begin{equation}\label{eieq}
 \Ric=-2g.
\end{equation}
The first-order variation of this equation along a path of metrics reads
(\cite[Theorem 1.174.d]{besse}):
\begin{equation}\label{lee}
-\delta^*(\delta+\tfrac{1}{2} d\Tr)\dot{g} +\tfrac12\left[\nabla^*\nabla \dot{g}
+\Ric\circ \dot{g}+\dot{g}\circ\Ric-2\circR \dot{g}\right]=-2\dot{g}\end{equation}
where the action of the curvature tensor $R$ on a symmetric $2$-tensor $h$ is defined as 
\[(\circR h)_{iq}=\sum_{j,k=1}^3 h_{jk}\langle R_{ij}v_k,v_q\rangle.
\]
Using \eqref{eieq}, equation \eqref{lee} is equivalent to
\[-\delta^*(2\delta+d\Tr)\dot{g} +\nabla^*\nabla \dot{g}-2\circR \dot{g}=0.
\]
A simple computation shows that for $g$ hyperbolic, 
\begin{equation}\label{rcirc}
\circR h =h - \Tr(h)g
\end{equation}
for every symmetric $2$-tensor $h$.

\subsection{Weitzenb\"ock formula for symmetric tensors}\label{sectw}

The following Weitzenb\"ock formulae hold for the rough Laplacian $\nabla^*\nabla$, the
twisted Hodge Laplacian $d^\nabla {d^\nabla}^*+{d^\nabla}^*d^\nabla$ and the Laplacian on functions
over a hyperbolic $3$-manifold: if $q_0$ is a traceless symmetric $2$-tensor and $a$ 
is a smooth function, then
\begin{align*}
&\nabla^*\nabla q_0= (d^\nabla {d^\nabla}^*+{d^\nabla}^*d^\nabla+3)q_0,\\
&\nabla^*\nabla(ag)=\Delta(a) g.
\end{align*}
Moreover, by \eqref{rcirc},
\begin{align*}
\circR (ag)={}&-2ag,& \circR q_0 = q_0.
\end{align*}

\subsection{The Laplace equation on $1$-forms}

Let $\Delta=\nabla^*\nabla$ be the rough Laplacian acting on $1$-forms 
(equivalently, on vector fields) on the compact manifold 
with boundary $K$. Clearly, $\Delta$ maps $C^\infty(K,TK)$ to itself.
Recall that $\nu$ is the unit outgoing vector field orthogonal 
to the boundary $\Sigma$ of $K$, and $L_\nu$ denotes the Lie derivative.

\begin{prop}\label{bvp}
The restriction
\[\Delta+2:\{V\in C^\infty(K,TK); V_{|\Sigma}\in T\Sigma,
L_\nu V \perp \Sigma\}\to C^\infty(K,TK)\]
is an isomorphism.
\end{prop}
\begin{proof}
Let $\cD$ denote the initial domain $\{V\in C^\infty(K,TK); V_{|\Sigma}\in T\Sigma,
L_\nu V \perp \Sigma\}$. Then by integration by parts using that $\Sigma$ is
totally geodesic, we have for all $V,V'\in\cD$:
\[
\langle \nabla^*\nabla V,V'\rangle = \langle \nabla V,\nabla V'\rangle = 
\langle V,\nabla^*\nabla V'\rangle. 
\]
This implies that $\nabla^*\nabla$ is symmetric and non-negative on $\cD$.
Its self-adjoint Friedrichs extension 
\[
\Delta_\cF:\cD_\cF\to L^2
\]
is therefore also non-negative, so
$\nabla^*\nabla+2:\cD_\cF\to L^2$ is invertible. By elliptic regularity,
the preimage of $C^\infty(K,TK)$ by this operator must lie in $\cD_\cF\cap C^\infty(K,TK)= \cD$.
\end{proof}

\section{The Hessian of the volume of compact hyperbolic $3$-manifolds with geodesic boundary}

\begin{theorem}\label{pozcomp}
Let $(K,g)$ be a compact hyperbolic $3$-manifold with totally geodesic boundary, 
and $\{g^s\}_{s\in\bR}$ a smooth family of hyperbolic metrics on $K$ with minimal boundary. 
Then the Hessian of the volume functional of $K$ at $g$ is positive.
\end{theorem}
\begin{proof}
For small $s$, the principal curvatures along $K$ are smaller than $1$, so equation
\eqref{fun} defines a funnel,
extending $g^s$ to a complete hyperbolic metric on $X=K\cup F$, unique up to isometry,. By hypothesis, $\Sigma$, the possibly disconnected
boundary of $K$, is minimal for each of the metrics $g^s$. Moreover, the outgoing normal geodesics to 
$\Sigma$ with respect to $g$ are also parametrized geodesics for $g^s$. By composing with 
a family of diffeomorphisms of $X$ preserving $\Sigma$, we can assume that $\dot{h}$, the first-order variation 
of the metrics $g^s$ restricted to $\Sigma$, is divergence-free. It follows that we can apply the results of Section \ref{cccm}, in particular $\dot{h}$ and $A$ are divergence-free, trace-free along $\Sigma$, while $\dot{g}$ is divergence-free, trace free on the funnel.

Consider the following boundary-value problem:
\begin{align}\label{Dp}
\begin{cases}
(\nabla^*\nabla+2) V=-(\delta +\tfrac{1}{2}d\Tr)\dot{g},\\
V\in C^\infty(K,TK),\\
V_{|\Sigma}\in T\Sigma,\\
L_\nu V \perp \Sigma.
\end{cases}
\end{align}
By Proposition \ref{bvp}, there exists a unique solution $V$ to \eqref{Dp}.
Set 
\[q:=\dot{g}+L_Vg.\]
If $\{\phi_s\}$ is the $1$-parameter group of diffeomorphisms of $K$ integrating $V$ 
(well-defined since $V$ is tangent to $\partial K$), then $q$ is the tangent vector field 
to the $1$-parameter family of metrics $G^s:=\phi_s^*g^s$. 

Remark that 
\begin{itemize}
 \item $G^s$ is hyperbolic;
\item  $\Vol(K,g^s)=\Vol(K,G^s)$;
\item  $\Sigma$ is minimal in $(K,G^s)$ for every $s$;
\item $\Sigma$ is totally geodesic at $s=0$;
\item $\nu^s:={\phi^s}^*\nu$ is the unit normal vector field to $\Sigma$ with respect to $G^s$.
\end{itemize}
The last property holds because $\nu$ is the unit normal vector field to $\Sigma$ with respect to $g^s$ for every $s$.

By the Schl\"afli formula of Rivin-Schlenker (see \cite{GMS}, Lemma 5.1), we have
\[\partial_s \Vol(K,g^s)=\tfrac{1}{2}\int_\Sigma (\Tr(\dA^s) + \tfrac{1}{2}\Tr((h^s)^{-1} \doh^s A^s))d\vol_{h^s}
=\tfrac{1}{8}\langle \doh^s, L_\nu g^s\rangle_{L^2(\Sigma, h^s)}.\]

The same formula for the family of metrics $G^s=\phi_s^*g^s$ gives
\begin{equation}
\partial_s \Vol(K,\phi_s^*g^s)=\tfrac{1}{8}\langle \partial_s G^s, L_{{\phi^s}^*{\nu}} G^s\rangle_{L^2(\Sigma,G^s)}.
\end{equation}
The term $L_{{\phi^s}^*{\nu}} G^s$, i.e., the second fundamental form of $\Sigma$ with respect to $G^s$, 
vanishes at $s=0$. One more derivative at $s=0$ shows therefore
\begin{equation}\label{ddv}
\partial_s^2 \Vol(K,G^s)_{s=0}=\tfrac{1}{8}\langle q, L_\nu q-L_{[V,\nu]}g \rangle_g.
\end{equation}

Since $\Vol(K,g^s)=\Vol(K,G^s)$, the above formula computes the second variation 
$\ddot\Vol$ of $\Vol(K,g^s)$.

\begin{theorem}\label{poz}
The inner product $\langle q, L_\nu q\rangle_{L^2(\Sigma,g)}$ is non-negative. Explicitly,
\[
 \langle q, L_\nu q\rangle_{L^2(\Sigma,g)}\geq \|q\|^2.
\]
\end{theorem}
\begin{proof}
The tensor $q$ is a solution to the linearized Einstein equation because the family $G^s$ 
consists of hyperbolic metrics. Use now the following identities on vector fields:
\begin{align*}(2\delta+d\Tr )\delta^*=\nabla^*\nabla+2,&&2\delta^*V = L_Vg.
\end{align*}
From \eqref{Dp}, it follows that $q=\dot{g}+L_Vg$ is in Bianchi gauge, i.e.,
\[(\delta+\tfrac{1}{2}d\Tr)q=0.\]
Equation \eqref{lee} implies that $q$ is a solution of the elliptic equation
\begin{equation}\label{elb}
(\nabla^*\nabla-2\circR)q=0.\end{equation} 
Decompose $q$ in its trace-free component
$q_0$ and its pure trace component $ag$ for some $a\in C^\infty(K)$.
Using the Weitzenb\"ock formulae from section \ref{sectw},
\eqref{elb} is equivalent to
\begin{align}\label{dirp}
\begin{cases}
(d^\nabla {d^\nabla}^*+{d^\nabla}^*d^\nabla+1)q_0=0,\\(\Delta+4) a=0.
\end{cases}
\end{align}
Because of these identities, integration by parts on $K$ gives 
\begin{align}
\int_\Sigma \langle q_0,L_\nu q_0\rangle d\vol_h
={}&\int_K(|d^\nabla q_0|^2+|{d^\nabla}^*q_0|^2+|q_0|^2)d\vol_g,\label{g1}\\
\int_\Sigma \langle a,L_\nu a\rangle d\vol_h={}&\int_K(|da|^2+4a^2)d\vol_g.\label{g2}
\end{align}
Since $L_\nu g=0$ (equivalent to $(\Sigma,h)\hookrightarrow (K,g)$ being totally geodesic), 
\eqref{g2} is the same as
\begin{align}
\int_\Sigma \langle ag,L_\nu (ag)\rangle d\vol_h={}&\int_K(|da|^2+4a^2)d\vol_g.\label{g3}
\end{align}
Since $\Tr(g^{-1}q_0)=0$ by definition and $L_\nu g=0$, it follows by applying $L_\nu$ 
that 
\[\Tr(g^{-1}L_\nu q_0)=0.\] 
Hence $L_\nu q_0$ is trace-free, $L_\nu(ag)$ is a multiple of $g$, and so
\eqref{g1} and \eqref{g3} give
\begin{align*}
\int_\Sigma \langle q,L_\nu q\rangle d\vol_h\geq \|q_0\|^2+4\|a\|^2\geq \|q\|^2.
\end{align*}
\qedhere
\end{proof}

Returning to \eqref{ddv}, we would like to analyze the remaining term. In fact we prove below that
it vanishes pointwise on $\Sigma$, thereby ending the proof of Theorem \ref{pozcomp}.
\end{proof}
\begin{prop}\label{lan}
The scalar product $\langle q,L_{[V,T]}g \rangle_g$ is pointwise zero on $\Sigma$.
\end{prop}
\begin{proof}
Let 
\[V=v_0+tu_1 T+t^2 (v_2+u_2T)+O(t^3)\] 
be the Taylor expansion of $V$ near $t=0$, using the fixed product decomposition near $\Sigma$. Here $v_0,v_2$ are vector fields on $\Sigma$, while $u_1,u_2$ are functions. Note that the coefficients $u_0$ and $v_1$ vanish (and we omit them from the formula) as a consequence of \eqref{Dp}.

From the definition, $q=\dg+L_Vg$ where we can also write $L_Vg=2(\nabla V)_\sym=2\delta^*V$. We recall that $\dg$ is tangential (it does not contain terms involving $dt$). The correction term equals at $\Sigma$
\begin{align}\label{lvg}
L_Vg=L_{v_0}h+u_1 dt\otimes dt +O(t)
\end{align} 
and so in particular it has no mixed terms of the type $dt\otimes \Lambda^1\Sigma$.
The vector field $[T,V]$ equals 
\[[T,V]=u_1T+2tv_2+2tu_2T +O(t^2).
\]
The tensor $L_{[T,V]}g$ does not have any tangential component at $t=0$. 
The mixed terms do not contribute in the scalar product with $q$ since that last tensor 
has no such mixed terms. The coefficient of $dt\otimes dt$ in $L_{[T,V]}g$ is $2u_2$.
But this term vanishes by the lemma below.
\end{proof}

\begin{lemma}
The second-order normal term $u_2$ in $V$ vanishes.
\end{lemma}
\begin{proof}
The free term $-2\delta \dot{g}-d\Tr(\dot{g})$ in the boundary-value problem \eqref{Dp} determining $V$ vanishes near $\Sigma$ by Lemma \ref{gTT}.
At $t=0$ the Hodge Laplacian $\Delta_H=dd^*+d^*d$
on $1$-forms takes the form
\[(\Delta_H V)_{|t=0}=\Delta^h v_0-2v_2-2u_2\nu. 
\]
Since $g$ is hyperbolic, Bochner's formula 
\[\Delta_H=\nabla^*\nabla+\Ric
\]
gives
$(\Delta_H+4)V=(\Delta+2)V=0$, and using that $V$ is tangent to $\Sigma$ we deduce $u_2=0$. 
\end{proof}

\section{The Hessian of the renormalized volume on the funnel}

We have seen above that $\vol(K)$ has positive Hessian at $g$. 
But since $\Sigma$ is minimal for $g^s$, we remark that 
\[\Volr(X,g^s;h_\infty^s)=\Vol(K,g^s).\] 
To prove Theorem \ref{main} 
we must therefore analyze the Hessian of $\Vol_R-\Vol(K)$.
For this, let $\omega^s\in C^\infty(\Sigma)$ be the conformal factor so that $e^{2\omega^s}h^s_\infty$ has constant curvature $-4$. 
Such a conformal factor is unique, smooth in $s$, and $\omega^0=0$. 

Following the proof of \cite[Theorem 9]{CM} one could obtain by similar methods:
\begin{prop}\label{pf}
Let $\omega^s$ be the conformal factor uniformizing the metrics $h_\infty^s$. Then for small $s$,
\[\Volr(X,g^s)\geq \Volr(X,g^s;h_\infty^s),\]
with equality at $s=0$
\end{prop}
As a consequence, $\ddot{V}_R\geq \frac{d^2}{ds^2}\Volr(X,g^s;h_\infty^s)_{|s=0}.$
However, rather than adapting the results of \cite{CM}, we prove below that
the functional $\Vol_R-\Vol(K)$ is convex at $g=g_\geod$, with an explicit lower bound.

The following Polyakov-type formula holds for the conformal variation of the renormalized volume
(cf.\ \cite{GMS}):
\begin{equation}\label{chvr}
\Volr(X,g^s; e^{2\omega^s}h^s_\infty)-\Volr(X,g^s; h^s_\infty)
=-\tfrac{1}{4}\int_\Sigma \left(|d\omega^s|^2_{h^s_\infty}+2\kappa_{h^s_\infty}\omega^s\right)
d\vol_{h^s_\infty}.
\end{equation}

Let $\ddot\kappa$ be the second variation of $\kappa_{h^s_\infty}$ at $s=0$:
\[
\ddot\kappa = \partial_s^2(\kappa_{h^s_\infty})_{|s=0}.
\]

\begin{lemma}
Let $\omega^s\in C^\infty(\Sigma)$ such that $\kappa_{e^{2\omega^s}h^s_\infty}=-4$. 
Then $\omega^s=\omega_2 s^2+O(s^3)$
with 
\[\omega_2=-\tfrac{1}{2}(\Delta_{h_\infty^0}+8)^{-1}\ddot\kappa.\]
\end{lemma}
\begin{proof}
The metrics $h^s$ are given by 
\[h^s_\infty=\tfrac{1}{4}h^s\left((1+A^s)^2\cdot,\cdot\right)
\]
hence the first-order variation is 
\begin{equation}\label{hid}
\doh_\infty=\tfrac{1}{4}(\doh+2h(\dA\cdot,\cdot)).
\end{equation}
By applying the formula \eqref{vark}, we get for the first-order variation of the 
Gaussian curvature of $h_\infty^s$:
\[2\dot{\kappa}_\infty=(\Delta_\infty+1)\Tr(\doh_\infty)+d^*\delta\doh_\infty.
\]
Both $\Tr(\doh_\infty)$ and $\delta\doh_\infty$ vanish by Lemma \ref{lTT}, 
thus $\dot{\kappa}_\infty=0$.

By the conformal change rule for the Gaussian curvature,
\begin{equation}\label{eccmp}
-4=\kappa_{e^{2\omega^s}h^s_\infty}=e^{-2\omega^s}(\kappa_{h_\infty^s}+\Delta_{h_\infty^s}\omega^s).
\end{equation}
Let $\omega^s=s\omega_1+s^2\omega_2+O(s^3)$ be the limited Taylor expansion of $\omega$. 
Write the expansion in $s$ up to errors of order $s^3$ of the above identity, 
using $\kappa_{h_\infty^s}=-4+O(s^2)$:
\begin{equation*}
-4=(1-2\omega_1s+O(s^2))\left(-4+O(s^2)+(\Delta_{h_\infty^0}+O(s))(s\omega_1+O(s^2))\right)
\end{equation*}
giving for the coefficient of $s$
\begin{align*}
\Delta_{h_\infty^0}\omega_1+8\omega_1=0.
\end{align*}
It follows by positivity of this elliptic operator on the closed surface 
$\Sigma$ that $\omega_1=0$, and returning to \eqref{eccmp},
\begin{equation*}
-4=(1-2\omega_2s^2+O(s^3))\left(-4+\tfrac{1}{2}\ddot\kappa s^2+O(s^3)
+(\Delta_{h_\infty^0}+O(s))(s^2\omega_2+O(s^3))\right).
\end{equation*}
The coefficient of $s^2$ must be $0$, hence
\[8\omega_2+\tfrac{1}{2}\ddot\kappa+\Delta_{h_\infty^0}\omega_2=0,\]
proving the lemma.
\end{proof}
Using this lemma, \eqref{chvr} gives for the quadratic term in the right-hand side:
\[
-\tfrac{1}{4}\int_\Sigma \left(|d\omega^s|^2_{h^s_\infty}+2\kappa_{h^s_\infty}\omega^s\right)d\vol_{h^s_\infty}
=-s^2\int_\Sigma (\Delta_{h_\infty^0}+8)^{-1}\ddot\kappa d\vol_{h^s_\infty} +O(s^3).
\]
By decomposing $\ddot\kappa$ in eigenmodes for $\Delta_{h_\infty^0}$, 
we see that in the integral the only surviving such term
is the zero-eigenmode, hence 
\begin{equation}\label{cor}
\Volr(X,g^s; e^{2\omega^s}h^s_\infty)-\Volr(X,g^s; h^s_\infty)
=-\tfrac{1}{8}s^2\int_\Sigma\ddot\kappa d\vol_{h^s_\infty} +O(s^3).
\end{equation}

\begin{lemma}\label{lvk}
The second variation of $\kappa_{h^s_\infty}$ equals $-8\Tr(\dA^2)$.
\end{lemma}
\begin{proof}
We have (see \cite{CM}, proof of Lemma 10):
\[\kappa_{h_\infty^s}=\frac{4\kappa_{h^s}}{2+\kappa_{h^s}}=4-\frac{8}{2+\kappa_{h^s}}.
\]
Thus, since $\partial_s\kappa_{h^s}|_{s=0}=0$, we get 
$\partial_s^2\kappa_{h_\infty^s}|_{s=0}=8\ddot\kappa_{h^s}|_{s=0}$. 

From the Gauss equation \eqref{gauss} and the fact that $\Tr(A^s)=0$ we deduce
\[\Tr((A^s)^2)=-2\kappa_{h^s}-2.\]
 Since $A=0$ at $s=0$, we get $\ddot\kappa_{h^s}|_{s=0}=-\Tr(\dA^2)$.
\end{proof}
\section{Proof of Theorem \ref{main}}
From \eqref{cor}, Theorem \ref{poz} and Lemma \ref{lvk}, we get 
\[\ddot{\Volr}(X,g^s)\geq \tfrac{1}{4}\int_\Sigma \Tr(\dA^2)d\vol_h +\int_K\|q\|^2.
\]
We would like to translate this into 
an inequality in terms of the Weil-Petersson metric on the Teichm\"uller space 
$\cT_\Sigma$. Consider a smooth slice
\[
\cg:\cT_\Sigma\to \cM_{-1}(X)
\]
with $\cg(h)=g_\geod$, where $h$ is the hyperbolic metric on $\Sigma$ 
viewed as the totally geodesic boundary of the convex core of 
$(X,g_\geod)$. We can assume, up to pulling back by a family of 
diffeomorphisms of $X$, that the boundary of $K$ is minimal with respect to
$g^s:=\cg(s)$ for every $s\in\cT$, and moreover that the geodesics normal to $\partial K$ are
the same for all $s$. Let $h^s$ be the metric
induced on $\Sigma=\partial K$ by restriction of $g^s$.
The tangent map to this restriction is the linear 
map 
\begin{align*}
R:T_h\cT\longrightarrow T_h\cT, &&\dg\longmapsto \doh 
\end{align*}
of finite-dimensional vector spaces.
By composing with the map $\doh\mapsto V$, where $V$ is the solution to the boundary problem
\eqref{Dp}, we obtain a linear map
\begin{align*}
Q:T_h\cT\longrightarrow T_{g_\geod} \cM_{-1}(K),&&  \dg\longmapsto q=\dg+L_V g.
\end{align*}

\begin{prop}
There exists a constant $c>0$, independent of the metrics $g^s$, such that
\[\|q\|_{L^2(K)}\geq c\|\dot{h}\|_{L^2(\Sigma)}.\]
\end{prop}
\begin{proof}
If $q=0$, it follows that $\doh+L_V h=0$ as a tensor on $\Sigma$. Since
$\doh$ is divergence-free, the tensors $\doh$ and $L_V h$ are orthogonal in
$L^2$, so they both must vanish. Hence, if $Q(\dg)=0$ we get that $R(\dg)=0$.
Therefore $R$ factors through $Q$, i.e., there exists a linear map
\begin{align*}
T:\mathrm{range}(Q)\longrightarrow T_h\cT,&&q\longmapsto \doh
\end{align*}
so that $T\circ Q=R$. The map must be bounded since $\cT$ is of finite dimension.
\end{proof}

\begin{remark}
On the quasi-fuchsian space, the Hessian of $\Volr$ at any point of the fuchsian locus equals
\[\ddot{\Volr}=\tfrac18 \|\dot{h}^+-\dot{h}^-\|^2,
\]
Thus the Hessian is only positive \emph{semi}definite in that case. It becomes however positive 
definite when we restrict 
the renormalized volume functional to a Bers slice, i.e., we keep fixed one conformal boundary.
In that case it equals $1/8$ of the Weil-Petersson metric, by a result of Krasnov and Schlenker 
\cite{KS08}. 
\end{remark}

\section{Proof of Theorem \ref{coro}}

Schlenker \cite[Theorem 1.1]{Sch} proved the following inequality between the renormalized volume and
the volume of the convex core of quasi-fuchsian manifolds:
\begin{align}\label{inerc}
\Volr(X,g^s)\leq \Vol(C(X,g^s)) - \tfrac14 L_m(l)
\end{align}
where $L_m(l)\geq 0$ is the length of the measured bending lamination
of the convex core. The proof consists in identifying the right-hand side
with the renormalized volume of $X$
starting from the boundary of the convex core. The induced metric at infinity
is Thurston's grafting metric on $\Sigma$, which is known to be of class $C^{1,1}$,
of non-positive curvature, and bounded below by 
the Poincar\'e (i.e., hyperbolic) metric on $\Sigma$. 
These properties imply \eqref{inerc}. 

This inequality carries over with the same proof
to convex co-compact manifolds, as was noted in \cite{BC}. In fact,  
Schlenker's proof remains valid for manifolds with funnels and with rank $2$-cusps
(i.e., geometrically finite without rank $1$-cusps). Hence we may safely use \eqref{inerc}
in our setting.

At the initial metric $g_\geod$ the two sets $K$ and $C(X,g_\geod)$ are the same, 
in particular they share the same volume. Moreover, $\Volr(X,g_\geod)=\Vol(K)$ since
the boundary of $K$ is totally geodesic with respect to $g_\geod$.
Since by Theorem \ref{main} 
the Hessian of $\Volr(X,g)$ is positive definite at $g_\geod$, the same can be said 
about the Hessian of the functional $\Vol(C(X,g))$ at $g_\geod$, using the inequality \eqref{inerc}.

\end{document}